\pdfoutput=1
\documentclass[12pt]{amsart}
\usepackage{listings,amsmath,amssymb,graphicx,url,ctable,graphics, setspace}
\usepackage[colorlinks]{hyperref}

\input{glyphtounicode}
\theoremstyle{definition}
\newtheorem*{Definition}{Definition}
\newtheorem*{Example}{Example}

\theoremstyle{plain}
\newtheorem*{Conjecture}{Conjecture}
\newtheorem{Theorem}{Theorem}
\newtheorem{Lemma}[Theorem]{Lemma}
\newtheorem{Corollary}[Theorem]{Corollary}

\newcommand{\fus}{\ensuremath{\mathcal{F}}}
\newcommand{\epsn}{\ensuremath{\varepsilon_0}}
\newcommand{\llimit}{\ensuremath{\mathcal{L}}}
\newcommand{\ulimit}{\ensuremath{\mathcal{U}}}
\newcommand{\Ord}{\text{Ord}}
\newcommand{\Num}{\text{Num}}
\newcommand{\exc}{\text{exc}}
\newcommand{\dup}{\text{dup}}

\title{Survey On Fusible Numbers}
\author{Junyan Xu}
\date{}
\begin{document}
\maketitle

\begin{abstract}
We point out that the recursive formula that appears in Erickson's presentation \emph{Fusible Numbers} is incorrect, and pose an alternate conjecture about the structure of fusible numbers. Although we are unable to solve the conjecture, we succeed in establishing some basic properties of fusible numbers. We suggest some possible approaches to the conjecture, and list further problems in the final chapter.
\end{abstract}

\section{Introduction}
The term ``fusible number'' was coined by Jeff Erickson in his presentation \cite{erickson}. Readers are encouraged to look at the presentation for some background and illustrations, while keeping in mind that not all of the results there are correct, as will be pointed out below.

To my mind the definition of fusible numbers is another example of simple construction that yields rich algebraic and (transfinite) combinatorial structure (another nice example would be Conway's $\mathbf{On_2}$ (cf. \cite{onag}, Chapter 6).
\begin{Definition}
The \emph{fuse operation} is the binary operation over real numbers defined by $a\sim b=(a+b+1)/2$, but only applicable when $|a-b|<1$. (We always assume the operands satisfy this restriction.) A \emph{fusible number} is a real number obtainable from 0 with the fuse operation. A valid expression $\alpha$ that only involves $0$ and $\sim$ is called a \emph{presentation} of its \emph{value} $v(\alpha)$. Evidently all fusible numbers are dyadic rationals. The set of all fusible numbers will be denoted \fus.
\end{Definition}
\begin{Example}
$9/8$ is a fusible number, and the valid expression $(0\sim0)\sim(0\sim(0\sim0))$ is a presentation of $9/8$, but it doesn't follows that $17/16=(0+9/8+1)/2$ is also fusible, because the condition $|0-9/8|<1$ is not satisfied. In fact, $17/16$ is not fusible. Similarly, $a\sim(a+1)=a+1$ is invalid, but one can safely construct $a+1$ by $(a\sim a)\sim(a\sim a)$.
\end{Example}
There are some basic properties of fusible numbers that are immediate corollaries of the inequality present in the definition:
\begin{Lemma}
\label{ineq}\
\begin{enumerate}
\item $a\sim b>\max\{a,b\}$, so all fusible numbers are $\ge0$.
\item $\min\{a,b\}+1/2\le a\sim b<\min\{a,b\}+1$.
\end{enumerate}
\end{Lemma}
We will prove that \fus\ is well-ordered (with order structure inherited from the real numbers) in the next section, so any subset of \fus\ contains a least element, and we can talk about the successor of a fusible number, or generally, the least fusible number greater than a particular real number.
\begin{Definition}
For $a\in\mathbb{R}$ we define $s(a)=\min\{b\in\fus:b>a\}$, $m(a)=s(a)-a$, and $\overline{a}=a+1$.
\end{Definition}
\begin{Example}
$m(0)=1/2$, $m(0.4)=0.1$, $m(1/2)=1/4$, $m(3/4)=1/8$, $m(1)=1/8$, and $m(2)=1/1024.$
\end{Example}
Erickson gives the following recursive formula for $m$.
\begin{equation}
\label{fml:eri}
m(x)=\begin{cases}
-x & \text{if}\ x<0\\
m(x-m(x-1))/2 & \text{otherwise}
\end{cases}
\end{equation}
Since $x=(x-1+m(x-1))\sim(x-m(x-1))=s(x-1)\sim(x-m(x-1))$, \eqref{fml:eri} is equivalent to the statement that, for every $a\in\fus$, all fusible number in the range $[\overline{a},\overline{s(a)})$ can be written as $s(a)\sim b$ for some $b\in\fus$. If $b$ is still in this range, we repeat this decomposition process to get the form $s(a)\sim(\cdots(s(a)\sim c))$, or $(s(a)\sim)^n c$, for some $c\in\fus\cap[\overline{a}-m(a),\overline{a})$.

Reversing this decomposition process allows us to construct new fusible numbers from smaller ones. Let $\fus'$ be the subset of $\fus$ that results from this process. We can determine the ordinal type of $\fus'$.

First we assign the ordinal $1$ instead of $0$ to the smallest fusible number $0$; this will only affect ordinals below $\omega$. It's an easy exercise to show the fusible numbers below $1$ are just $1-2^{-n}$ ($n\in\mathbb{N}$), which is also consistent with \eqref{fml:eri}. So the ordinal type of $\fus'\cap(-\infty,1)$ is $\omega$, and we have $\Ord(1)=\omega=\omega^{\Ord(0)}$. Now we use transfinite induction to show $\Ord(\overline{a})=\omega^{\Ord(a)}$ for $a\in\fus'$.
\begin{enumerate}
\item For every $a\in\fus'$, we have $\Ord(s(a))=\Ord(a)+1$. Let $\Ord(a)=\alpha$, then by induction hypothesis we have the ordinal type of $\fus'\cap(-\infty,\overline{a})$ is $\omega^\alpha$. Since an ending segment of $\omega^\alpha$ is isomorphic to $\omega^\alpha$ itself (which can be proved using Cantor's normal form), the ordinal type of $\fus'\cap[\overline{a}-m(a),\overline{a})$ is also $\omega^\alpha$. Note that $s(a)\sim$ carries the interval $[\overline{a}-m(a),\overline{a})=[\overline{s(a)}-2m(a),\overline{s(a)}-m(a))$ into $[\overline{s(a)}-m(a),\overline{s(a)}-m(a)/2)$ and then into $[\overline{s(a)}-m(a)/2,\overline{s(a)}-m(a)/4)$ and so on. So if we denote the interval \label{ian} $[\overline{s(a)}-2^{1-n}m(a),\overline{s(a)}-2^{-n}m(a))$ by $\mathcal{I}_{a,n}$ for $n\in\mathbb{N}^+$, then the ordinal type of $\fus'\cap\mathcal{I}_{a,n}$ is $\omega^\alpha$, and $[\overline{a},\overline{s(a)})=\bigcup_{n\in\mathbb{N}^+}\mathcal{I}_{a,n}$. So the ordinal type of $\fus'\cap(-\infty,\overline{s(a)})$ is $\omega^\alpha\cdot\omega=\omega^{\alpha+1}$, and we have $\Ord(\overline{s(a)})=\omega^{\Ord(s(a))}$.
\item For every $a\in\fus'$ with $\Ord(a)=\alpha$ a limit ordinal, by induction hypothesis we have $\Ord(\overline{b})=\omega^{\Ord(b)}$, if $\Ord(b)<\Ord(a)$. So the ordinal type of $\fus'\cap(-\infty,\overline{a})=\bigcup\{\fus'\cap(-\infty,\overline{b}):b\in\fus',b<a\}$ is $\sup\{\omega^\beta:\beta<\alpha\}=\omega^\alpha$.
\end{enumerate}

Now since $\Ord(1)=\omega$, we have $\Ord(2)=\omega^\omega$, $\Ord(3)=\omega^{\omega^\omega}$ and so on. So the ordinal type of $\fus'$ is $\sup\{\omega,\omega^\omega,\omega^{\omega^\omega},\dots\}=\epsn$. Moreover, if $a\in\fus'$ and $\alpha=\Ord(a)$ is a limit ordinal, \begin{equation}
\label{fun}
\alpha'[n]=\Ord(a-2^{1-n}m(a))
\end{equation}
naturally defines a fundamental sequence for $\alpha$, which agrees with the canonical choice of fundamental sequence $\alpha[n]$ (cf. \cite{sladek}, Definition 20, where it is denoted $d_n[\alpha]$) except for an excess $\exc(\alpha)$ that depends on $\alpha$ in the sense that $\alpha'[n]=\alpha[n+\exc(\alpha)]$. These observations suggest some connections with proof theory, which will be further investigated.
\\

Erickson provides no justifications for \eqref{fml:eri}, i.e. the claim that $\fus'$ coincides with $\fus$, and unfortunately, the formula turns out to be incorrect. Since \eqref{fml:eri} gives $m(31/16)=2^{-11}$, we have $19/16\sim s(31/16)=33/16+2^{-12}\in\fus$, so $m(33/16)\le2^{-12}$. But \eqref{fml:eri} gives $m(33/16)=2^{-11}$. This counterexample tells us that the process described above doesn't produce all of the fusible numbers. Nonetheless we still know the ordinal type of \fus\ is at least \epsn, since by no means might a larger ordinal embed in a smaller one.

So where does the missing fusible numbers come from? I perform more calculations and came up with the following conjecture.
\begin{Conjecture}[Main Conjecture]
For every $a\in\fus$, write $[\overline{a},\overline{s(a)})=\bigcup_{n\in\mathbb{N}^+}\mathcal{I}_{a,n}$ as \hyperref[ian]{before}. Then the fusible numbers inside $\mathcal{I}_{a,n}$ can be written as $s^n(a)\sim c$ for some $c\in\fus$. In other words, $\mathcal{I}_{a,n}\cap\fus$ is a translated copy of $[\overline{a}-2^{1-n}m(a),\overline{a})\cap\fus$ scaled by a factor of $1/2$.
\end{Conjecture}
In fact, for $a\in\fus$, we have $s^n(a)=a+(2-2^{1-n})m(a)$ and $s^n(a)\sim\overline{a}=\overline{s(a)}-2^{-n}m(a)$ by the equality $m(s(a))=m(a)/2$ which will be proved in the next section.
\begin{figure}[h]
\includegraphics[width=297pt]{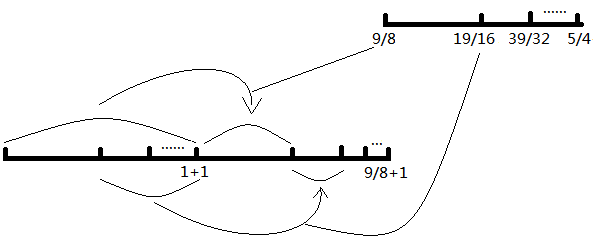}
\caption{Illustration of the conjecture when $a=1$.}
\end{figure}

The recursive formula equivalent to this conjecture is
\begin{equation}
\label{fml:my}
m(x)=\begin{cases}
-x & \text{if}\ x<0\\
m(x-a-1/d+2^{\lceil \log_2a\rceil})/2 & \text{otherwise}
\end{cases}
\end{equation}
where $a=m(x-1)$ and $d$ is the denominator of $s(x-1)$, a power of 2. In fact, $1/d$ is the gap between $s(x-1)$ and the previous fusible number, and $2^{\lceil \log_2a\rceil}$ is the least power of $2$ greater than or equal to $a$.

The ordinal type of the \fus\ would still be \epsn\ if Main Conjecture were true, since the ordinal type of $\fus\cap\mathcal{I}_{a,n}$ would still be $\omega^{\Ord(a)}$.

A proof of the conjecture is desired, since a lot of problems about fusible numbers seem intractable unless we assume the conjecture.\\

Starting from 33/16, the result of Erickson's formula diverge from the value given by Main Conjecture further and further. The difference soon becomes enormous, as demonstrated in the following table.
\ctable[
caption=Values of $-\log_2m(3-2^{-n})$ according to Erickson's formula \eqref{fml:eri} and Main Conjecture \eqref{fml:my}.,
pos=h
]{ccccccccc}
{\tnote[*]{$\uparrow$ is Knuth's up-arrow: $a\uparrow b=a^b$ and $a\uparrow^{n+1}b=(a\uparrow^n)^{b-1}a$.}}
{\FL
$n$&1&2&3&4&...&10&&$n$\ML
\eqref{fml:eri}&31&112&503&18443&...&1541023936&&$\sim\!e(n+2)!$\NN
\eqref{fml:my}&51&48804&
$>\!\!2\!\uparrow\!\uparrow\!9$\tmark[*]&
$>\!\!2\!\uparrow\!\uparrow\!\uparrow\!10$&...&
$>\!\!2\!\uparrow^9\!16$&&
$>\!\!2\!\uparrow^{n-1}\!\!(n+6)$\LL}

\section{Theorems and Proofs}
Below are some fundamental properties of the set of fusible numbers, but the proofs are not so easy. If Main Conjecture is true, these facts will become evident.

\begin{Theorem}
\label{wellord}
\fus\ is well-ordered.
\end{Theorem}
\begin{proof}
Consider a (strictly) decreasing sequence $\{a_n\}\subset\fus$. All $a_n>0$ since no fusible numbers are less than $0$, so $\{a_n\}$ is bounded from below and converges to some real number. Let \llimit\ be the set of all such real numbers. Then \llimit\ is bounded from below, and the limit of a decreasing sequence in \llimit\ is again in \llimit. So if $\llimit\neq\varnothing$, $l:=\inf\llimit\in\llimit$.

If \fus\ is not well-ordered, there will be a decreasing sequence in \fus, so $\llimit\neq\varnothing$. Let $\{a_n\}\subset\fus$ be a decreasing sequence such that $\lim a_n=l$. Since all $a_n>0$, we can write $a_n=b_n\sim c_n$ ($b_n\le c_n$, $b_n, c_n\in\fus$). $\{b_n\}$, as a sequence of real numbers, contains either a decreasing subsequence or a nondecreasing one. By passing to subsequences, we may assume $\{b_n\}$ is either decreasing or nondecreasing.

Since $b_n\le a_n-1/2$ by Lemma \ref{ineq}, if $\{b_n\}$ is decreasing, we have $\lim b_n\in\llimit$ and $\lim b_n<\lim a_n=\inf\llimit$, which is a contradiction.

If $\{b_n\}$ is nondecreasing, $\{c_n\}$ will be decreasing. Since $c_n<a_n$, we have $\lim c_n\le\lim a_n=\inf\llimit$, so $\lim c_n=\lim a_n$ and $\lim b_n=\lim c_n-1$. Since $\{b_n\}$ is nondecreasing and $\{c_n\}$ is decreasing, we find $c_n-b_n>1$, contrary to the restriction of the fuse operation.
\end{proof}

This proof depends on the \emph{Axiom of Dependent Choice} ($\mathsf{DC}$), which makes it a bit unsatisfactory. It seems likely that $\mathsf{DC}$ can be circumvented, but finding a proof within \emph{Peano Arithmetic} ($\mathsf{PA}$) is impossible since $\mathsf{PA}$ does not prove the well-orderedness of $\epsn$ (cf. \cite{sladek}, Section 5).\\

Exploiting well-orderedness, we have the following algorithm that calculates the function $m$ without assuming Main Conjecture, which however is quite inefficient.

\begin{quote}
\begin{lstlisting}[language=Mathematica, basicstyle=\footnotesize]
(*coded in Mathematica*)
$RecursionLimit = Infinity;
m[x_] := Module[{y, v = 2 x - 1, d, e},
   If[x < 0, -x,
    y = v - p[x - 1]; d = m[y]; y = y + d;
    While[y = y - 1/Denominator[y]; 2 y > v,
     y = v - p[v - y]; e = m[y];
     d = Min[d, e]; y = y + e];
    d/2]];
p[x_] := x + m[x];
(*zigzag algorithm*)
\end{lstlisting}
\end{quote}

\begin{Theorem}
\label{uplim}
The limit of any bounded nondecreasing sequence in \fus\ is still in \fus.
\end{Theorem}
\begin{proof}
Define $\ulimit=\{a\in\mathbb{R}:a\notin\fus$ and $a$ is the limit of a nondecreasing sequence $\subset\fus\}$. There is no decreasing sequence in \ulimit\ since such a sequence will also give us a decreasing sequence in \fus. So if $\ulimit\neq\varnothing$, we have $l:=\inf\ulimit\in\ulimit$.

Let $\{a_n\}$ be a nondecreasing sequence in \fus\ such that $\lim a_n=l$. Again we write $a_n=b_n\sim c_n$ ($b_n\le c_n$, $b_n, c_n\in\fus$). Since there are no decreasing sequence in \fus\ by Theorem \ref{wellord}, we may assume $\{b_n\}$ and $\{c_n\}$ are nondecreasing by choosing suitable subsequences.

If $\lim c_n=\lim a_n$, we have $\lim b_n=\lim a_n=l-1\notin\ulimit$, so $\lim b_n\in\fus$ and $l=\lim b_n+1\in\fus$. Contradiction.

If $\lim c_n<\lim a_n$, we have $\lim b_n\le\lim c_n<l$, so $\lim b_n, \lim c_n\in\fus$, and $\lim c_n<\lim b_n+1$. So $l=\lim b_n\sim\lim c_n\in\fus$. Contradiction.
\end{proof}

\begin{Corollary}
\fus\ is a closed subset of $\mathbb{R}$.
\end{Corollary}

For the next lemma we need to introduce the notions of depth and exponent. Please also recall the notions of (valid) expression, value and presentation introduced in the first section.

\begin{Definition}
The depth $d(\alpha)$ of an expression $\alpha$ that only involves only $0$ and $\sim$ is defined recursively: $d(0)=0$ and $d(\alpha\sim\beta)=\max\{d(\alpha),d(\beta)\}+1$. If one regards an expression as a binary tree, then its depth is just the height of the corresponding tree.

The depth $d(a)$ of a fusible number $a$ is the maximal depth among its presentations' depths.

A dyadic rational $a$ other than $0$ can be uniquely written as $(2k+1)/2^n (k,n\in\mathbb{Z})$, and we define $e(a)=n$, called the \emph{exponent} of $a$. $e(0)$ is defined to be $-\infty$.
\end{Definition}

\begin{Lemma}
$d(\alpha)\ge e(v(\alpha))$.
\end{Lemma}
\begin{proof}
If $d(\alpha)=0$, we have $\alpha=0$, $v(\alpha)=0$, and $e(v(\alpha))=-\infty$, so the lemma holds.
Let $n>0$, and suppose for all expressions $\beta$ with $d(\beta)<n$ the lemma holds. For any expression $\alpha$ with $d(\alpha)=n$, write $\alpha=\beta\sim\gamma$, $d(\beta)=n-1$ and $d(\gamma)<n$, we have that $e(v(\beta)),e(v(\gamma)),e(1)=0$ are all less than $n$, so $e(v(\alpha))=e(v(\beta)+v(\gamma)+1)+1\le n=d(\alpha)$. By induction this completes the proof.
\end{proof}

If $\alpha$ is a valid expression, of course $v(\alpha)\in\fus$. Moreover the following is true.
\begin{Lemma}[forward and backward]\
\begin{enumerate}
\item If $\alpha$ is a valid expression, $d(\alpha)=n$, then $v(\alpha)+2^{-n-1}\in\fus$.
\item If in addition $\alpha$ is not $0$, then $v(\alpha)-2^{-n}\in\fus$ with a presentation of depth $\ge n-1$.
\end{enumerate}
\end{Lemma}
\begin{proof}
We proceed by induction on $n$.
\begin{enumerate}
\item The only depth-$0$ expression is $0$, and we have $0+2^{-0-1}=1/2=0\sim0\in\fus$, so the lemma holds for $n=0$. Assume for every valid depth-$n$ expression $\beta$, we have $v(\beta)+2^{-n-1}\in\fus$. Then for any valid depth-$(n+1)$ expression $\alpha$, we write $\alpha=\beta\sim\gamma$, $|b-c|<1$, $d(\beta)=n$ and $d(\gamma)\le n$, hence $e(|b-c|)\le n$ by the previous lemma, so we in fact have $|b-c|\le1-2^{-n}$. (Here $b=v(\beta)$ and $c=v(\gamma)$.) Since $b+2^{-n-1}\in\fus$ and $|b+2^{-n-1}-c|\le1-2^{-n-1}<1$, we obtain $v(\alpha)+2^{-n-2}=(b+2^{-n-1})\sim c\in\fus$.
\item For $\alpha\neq0$, we have $d(\alpha)\ge1$. The only depth-$1$ expression is $0\sim0$ with value $1/2$, and we have $1/2-2^{-1}=0\in\fus$ with the presentation $0$ of depth $0$. Assume for every valid depth-$n$ expression $\beta$, we have $v(\beta)-2^{-n}\in\fus$ with a presentation of depth $\ge n-1$. Then for any valid depth-$(n+1)$ expression, again we have $\alpha=\beta\sim\gamma$, $|b-c|\le1-2^{-n}$, so $-1\le b-2^{-n}-c\le1-2^{1-n}$. If $b-2^{-n}-c=-1$, then $a-2^{-n-1}=(b-2^{-n})+1\in\fus$ has a presentation of depth $\ge n+1$; otherwise $a-2^{-n-1}=(b-2^{-n})\sim c\in\fus$ has a presentation of depth $\ge n$.
\end{enumerate}
\end{proof}

Let $a$ be a fusible number, Choose a presentation $\alpha$ of $a$, we have $d(\alpha)\ge e(a)$ and $a+2^{-d(\alpha)-1}\in\fus$, so $m(a)\le2^{-e(a)-1}$, which implies $e(s(a))>e(a)$, which in turn implies $a+2^{-e(s(a))}\le s(a)$. Suppose $\beta$ is a presentation of $s(a)$, then $s(a)-2^{-d(\beta)}\in\fus$, so it can't be greater than $a$ because there are no fusible number between $a$ and $s(a)$. So $a+2^{-e(s(a))}\ge a+2^{-d(\beta)}\ge s(a)$. We thus obtain that $m(a)=2^{-e(s(a))}$ is a (negative) power of $2$, and that the the depth of any presentation of $s(a)$ is equal to its exponent $e(s(a))$. So $d(s(a))=e(s(a))$ and $a=s(a)-2^{-d(s(a))}$ has a presentation with depth $\ge d(s(a))-1$, and cannot have a presentation with depth $>d(s(a))-1$, so by definition we have $d(a)=d(s(a))-1$. Thus we obtain
\begin{Theorem}
If $a\in\fus$, we have $m(a)=2^{-d(a)-1}$ and $m(s(a))=m(a)/2$.
\end{Theorem}

\section{Further Problems and Topics}
Besides the Main Conjecture, there are other interesting problems and topics awaiting solution, some of which I'd like to list here. They may be looked at from algebraic, combinatorial, or number-theoretic perspective. One may assume the Main Conjecture to proceed or first try to resolve the Main Conjecture.
\begin{enumerate}
\item Find an algorithm that translates a fusible number to its ordinal. If the Main Conjecture is true, this is easy. Furthermore, describe the fuse operation in terms of ordinal arithmetic.
\item The Main Conjecture may proved by showing the subset of \fus\ constructed according to the conjecture is closed under fuse operation, possibly via ordinal arithmetic; or by exhibiting a procedure that can transform arbitrary valid expression into the normal form described in the conjecture.
\item If we take the view that an expression of a fusible number is a binary tree, we find that an internal node at depth $n$ contributes $2^{-n-1}$ to its value and a leaf node contributes nothing. The condition $|a-b|<1$ demand the tree to be somewhat ``balanced''. How can we extract some information about fusible numbers from such consideration?
\item The function $f(n)=-\log_2m(n)\ (n\in\mathbb{N})$ grows quite rapidly. Determine how fast it grows in terms of some growth hierarchy.
\item Moreover, $f_\alpha(n)=-\log_2m(\Num(\omega^\alpha\cdot n))$ (where $\Num(\alpha)$ denotes the fusible number corresponding to the ordinal $\alpha$) naturally defines a growth hierarchy, and the function $f$ above may be considered as $f_{\varepsilon_0}$ in this hierarchy. Assuming the Main Conjecture, the functions in this hierarchy satisfies the following recursive relations similar to the defining relations of Hardy hierarchy:
\begin{align*}
& f_0(n)=n\\
& f_{\alpha+1}(n)=f_{\alpha}(n+1)+1\\
& f_{\alpha}(n)=f_{\alpha'[n]}(1) +1\ \ \ \text{ (if }\alpha\text{ is a limit)}
\end{align*}
where the fundamental sequence $\alpha'[n]$ is defined by \eqref{fun}. How is this hierarchy related to other growth hierarchies?
\item It's easy to show that $\lim_{x\to+\infty}m(x)=0$, so we can define the sequence $g(n)=\min\{x\in\mathbb{R}:(\forall y>x)\ m(y)<2^{-n}\}=\max\{a\in\fus:d(a)=n-1\}\ (n\in\mathbb{N}^+)$, the first few terms of which are $0, 1/2, 1, 5/4, 3/2, 13/8, 7/4, 29/16, 15/8, 2, 129/64, 33/16$. $g$ is in some sense the inverse function of $f$.
\item Define the duplicate number of a fusible number $a$ by $\dup(a)=\#\{(b,c)\in\fus\times\fus:a=b\sim c\text{ and }b\le c\}$, which is always finite thanks to the well-orderedness of \fus. Investigations on its properties may provide insights to the solution of the conjecture.
\item There are other algebraic operations on fusible numbers. For example, $(a,b)\mapsto a+b$ maps $\fus\times\fus$ to \fus, and $a\mapsto 2a-1$ maps $\fus\backslash\{0\}$ to \fus. 
\item The fuse operation may be replaced by other two-variable function satisfying certain conditions to yield other sets of fusible numbers which are still well-ordered. For example, one may instead defines $a\sim b=(na+mb+1)/(m+1)$ ($n\ge1, m\in\mathbb{R}$ are constants) applicable when $a\le b<na+1$.
\end{enumerate}

\bibliography{texify}
\bibliographystyle{amsalpha}

\end{document}